\documentclass[a4paper, 12pt, final]{article}

\newcommand{\version}{3.6}

\usepackage[intlimits, nosumlimits]{amsmath}
\usepackage{amssymb,amsthm}
\usepackage{xspace}

\usepackage{comment}

\usepackage{ifpdf}
\ifpdf 
\usepackage[pdfborder={0 0 0} ]{hyperref}
\else
\usepackage{hyperref}
\fi

\author{A. Fornasiero \&\ E. Vasquez Rifo}

\title{Hausdorff measure on o-minimal structures\\
\normalsize{Version \version}}

\ifpdf 
  \hypersetup{%
    pdfauthor={Antongiulio Fornasiero \& Elisa Vasquez Rifo},
    pdfsubject=%
     {Hausdorff measure of definable subsets of an o-minimal structure},
    pdfkeywords=%
     {Hausdorff measure, o-minimality, Cauchy-Crofton formula, Whitney arc property},
  }
\fi

\makeatletter
\renewcommand\@makefnmark{{\normalfont({\scriptsize\@textsuperscript\@thefnmark})}}
\makeatother

\newcommand*{\Pa}[1]{\bigl( #1 \bigr)}
\newcommand*{\pair}[1]{\langle #1 \rangle}
\newcommand*{\set}[1]{\{#1\}}

\newcommand*{\abs}[1]{\lvert#1\rvert}
\newcommand*{\norm}[1]{\lVert#1\rVert}

\newcommand{\Nat}{\mathbb{N}}
\newcommand{\Real}{\mathbb{R}}
\newcommand{\Rb}{\bar \Real}
\newcommand{\Q}{\mathbb{Q}}
\newcommand{\K}{K}
\newcommand{\Kb}{\mathring \K}
\newcommand{\D}{D}

\newcommand{\Cone}{\mathcal C^1}
\newcommand{\la}{\langle}
\newcommand{\ra}{\rangle}
\DeclareMathOperator{\End}{End}
\DeclareMathOperator{\DIFF}{D}
\DeclareMathOperator{\dom}{dom}
\newcommand{\Diff}{\mathord{\DIFF}}
\newcommand{\Jac}{J}
\newcommand{\Je}{J_e}
\newcommand{\de}{\:\mathrm d}
\newcommand{\Haus}{\mathcal H}
\newcommand{\He}{\Haus^{e}}
\newcommand{\HeK}{\Haus^{e}}
\newcommand{\HnK}{\Haus^{n}}
\newcommand{\Leb}{\mathcal L}
\newcommand{\LEK}{\mathcal L^E}
\newcommand{\LnK}{\Leb^n}
\newcommand{\LeK}{\Leb^e}
\newcommand{\Agr}{\mathcal{AG}}
\DeclareMathOperator{\st}{st}

\newcommand*{\cl}[1]{\overline{#1}}
\DeclareMathOperator{\rank}{rank}
\newcommand{\comp}{\circ}
\DeclareMathOperator{\Vol}{Vol}
\DeclareMathOperator{\length}{length}
\newcommand{\x}{\bar x}
\newcommand{\bcup}{\textstyle\bigcup}
\newcommand{\BAD}{\mathsf{BAD}}
\DeclareMathOperator{\Imag}{Im}

\newcommand{\bilip}{bi-Lipschitz\xspace}

\newcommand{\berect}{basic $e$-rectifiable\xspace}
\newcommand{\zsalg}{$\emptyset$-semialgebraic\xspace}
\newcommand{\cf}{cf\mbox{.}\xspace}
\newcommand{\Wlog}{W.l.o.g\mbox{.}\xspace}
\newcommand{\wloG}{w.l.o.g\mbox{.}\xspace}

\newcommand{\rom}{\textup}
\def\hyph{\nobreakdash-\hspace{0pt}\relax}

\newtheorem{lem}{Lemma}[section]
\newtheorem{thm}[lem]{Theorem}
\newtheorem{cor}[lem]{Corollary}

\newtheorem{open problem}[lem]{Open problem}

\theoremstyle{remark}

\newtheorem{claim}{Claim}
\newtheorem*{claim*}{Claim}

\theoremstyle{definition}
\newtheorem{dfn}[lem]{Definition}

\newtheorem{rem}[lem]{Remark}
\newtheorem{final remark}[lem]{Final remark}
\newtheorem{example}[lem]{Example}

\newenvironment{sentence}[1][]{%
  \begin{list}{}{%
    \setlength\topsep{0.5ex}%
    \setlength\leftmargin{\parindent}%
  }%
  \item[#1]
 }
 {\end{list}}

\begin{document}
\maketitle

\begin{abstract}
We introduce the Hausdorff measure for definable sets in an o-minimal
structure, and prove the Cauchy-Crofton and co-area formulae for the
o-minimal Hausdorff measure.
We also prove that every definable set can be partitioned into ``basic
rectifiable sets'', and that  the Whitney arc property holds for basic rectifiable sets.
\end{abstract}



\section{Introduction}
Let $\K$ be an o-minimal structure expanding a field.
We introduce, for every $e \in \Nat$, the 
$e$-dimensional Hausdorff measure for definable sets, which is the
generalization of the usual Hausdorff measure for real sets \cite{morgan}.
We also show that every definable set can be partitioned into ``basic
$e$-rectifiable sets'' (\S \ref{sec:partition}). 
Moreover, we generalize some well known result from geometric measure theory,
such as the Cauchy-Crofton formula (which computes the Hausdorff measure of a
set as the average number of points of intersection with hyperplanes
of complementary dimension) and the co-area formula (a generalization of
Fubini's theorem), to the o-minimal context.

The measure defined in \cite{bo} is the starting point for our construction
of the Hausdorff measure.  A theorem of \cite{bp} allows us to prove that
integration using the Berarducci-Otero measure
satisfies properties analogous to the ones for
integration over the reals (for example, the change of variable formula).  
If $\K$ is sufficiently saturated, the Berarducci-Otero measure of a bounded
definable set $X$ is $\Leb_\Real(\st(X))$, where $\Leb_\Real$ is the Lebesgue
measure and $\st$ is the standard-part map.
However, the naive definition of Hausdorff measure given by
\begin{equation}
\label{eq:naive}
\Haus^e(X) := \Haus^e_\Real(\st(X))
\end{equation}
does not work (because the resulting ``measure'' is not additive: see
Example~\ref{ex:double}). 
The correct definition for the $e$-dimensional Hausdorff measure  is 
defining it first for basic $e$-rectifiable sets via \eqref{eq:naive},
and then extending it to definable sets by using a partition into basic
$e$-rectifiable pieces. Such a partition is obtained by using partitions into
$M_n$-cells (\cite{kurdyka}, \cite{pawlucki}, \cite{evr}), a consequence of
which is the Whitney arc property for basic $e$-rectifiable sets~(\S\ref{sec:whitney}).



\section{Lebesgue measure on o-minimal structures}\label{sec:Lebesgue}
The definitions of measure theory are taken from~\cite{halmos}.

Let $\Rb := \Real \cup \set{\pm \infty}$ be the extended real line.
Let $\K$ be a $\aleph_1$-saturated o-minimal structure, expanding a field.
Let $\Kb$ be the set of finite elements of~$\K$.
Let $\st: \K^n \to \Rb^n$ be the function mapping  $\x$ to the $n$-tuple of 
standard  parts of the components of~$\x$.

For every $n \in \Nat$, let $\Leb^n_\Real$ be the $n$-dimensional Lebesgue
measure (on~$\Real^n$). 
If $n$ is clear from context we drop the  superscript. 
Let $\LnK_1$ be the o-minimal measure on $\Kb^n$ defined in~\cite{bo}. 
More precisely, $\LnK_1$ is a measure on the $\sigma$-ring $R_n$ generated by
the definable subsets of $\Kb^n$; thus, $(\Kb^n, R_n, \LnK_1)$ is a measure
space.
Moreover, since $\Kb^n \in R_n$, $R_n$ is actually a $\sigma$-algebra.

Notice that $\LnK_1$ can be extended in a natural way to a measure
$\LnK_2$ on the $\sigma$-ring $\mathcal B_n$ 
generated by the definable subsets of $\K^n$ of finite diameter.  
Finally, we denote by $\LnK$ the completion of~$\LnK_2$, 
and if $n$ is clear from context we drop the superscript.
Notice that the $\sigma$-ring $\mathcal B_n$
is not a $\sigma$-algebra.

\begin{rem}[{\cite[Thm.~4.3]{bo}}]\label{rem:BO}
If $C \subset \Kb^n$ is definable,
then $\Leb^n(C)$ is the Lebesgue measure of $\st(C)$ .
\end{rem}

\begin{dfn}
For $A \subseteq \K^n$ and  $f: \K^n \to \K^m$ we  define
$\st(f): A \to \Rb^m$ by $\st(f)(x) = \st(f(x))$.
\end{dfn}

\begin{rem}
If $A\subseteq\Kb^n$ and $f : A \to \K$ are definable,
then $\st(f)$ is an $\Leb^n$\hyph measurable function.
\end{rem}

\begin{dfn}
Let $A \subseteq \Kb^n$ and $f: A \to \K$ be definable.
If $\st(f)$ is $\Leb^n$-integrable  we will denote its integral by
\[
\int_A f \de \Leb^n;\quad \int_A f(x) \de x;\quad \int_A f(x) \de
\Leb^n(x)\quad\text{or }\int_A f.
\]
\end{dfn} 

\begin{rem}
If $A \subseteq \Kb^n$ and $f: A \to \Kb$ are definable, then
$\st(f)$ is $\Leb$\hyph integrable.
\end{rem}

Let $\Real_K$ be the structure on $\Real$ generated by the sets of the form 
$\st(U)$, where $U$ varies among the definable subsets of~$\K^n$.
By~\cite{bp}, $\Real_K$ is o-minimal.
\begin{rem}\label{rem:HPP}
Let $U \subseteq \Kb^n$ be definable.
Then,
$\dim(\st(U)) \leq \dim(U)$.
\end{rem}
\begin{proof}
Let $\dim(U) = d$.
After a cell decomposition, we can assume that $U$ is the graph of a definable
continuous function $f: V \to \Kb^{n - d}$, with $V \subset \Kb^d$ open cell.
We can then conclude by applying the method in \cite[Lemma 10.3]{HPP}.
\end{proof}

\begin{dfn}
A function  $f$ is Lipschitz if there is $C\in\Kb$ such that, 
for all $x,y\in\dom(f)$, we have $|f(x)-f(y)| < C|x-y|$
(notice the condition on $C$ being finite).
An invertible function $f$ is bi-Lipschitz if  both $f$ and $f^{-1}$ are
Lipschitz. 
\end{dfn}

\begin{rem}\label{rem:volume}
Let $U \subset \Kb^n$ and $f: U \to \Kb$ be definable, with $f \geq 0$.
Then,
\[
\int_U f \de \Leb^n = \Leb^{n+1}\Pa{\set{\pair{\x, y} \in U \times \K: 0 \leq
y \leq f(\x)}}.
\]
\end{rem}

\begin{lem}[Change of variables]\label{lem:change-variables}
Let $U, V \subseteq \Kb^n$ be open and definable, 
and let $A \subseteq U$ be definable. 
Let $f: U \to V$  be definable and \bilip and $g: V \to \Kb$  be definable, then
\[
\int_{f(A)} g = \int_A \abs{\det \Diff f}\; g \comp f .
\]
\end{lem}
Before proving the above lemma, we need some preliminary definitions and
results.

\begin{lem}\label{lem:int-Stan}
Let $U \subset \Kb^n$ be open and let $f: U \to \Kb$ be definable. 
Then there is a  $\Real_K$-definable function $\cl f:C\to\Real$, where $C\subset\st(U)$ is an open set, such that 
\begin{enumerate}
\item[i)] $E:= \Pa{\st(U)\setminus C} \cup \Pa{C \cap \st(\K^n \setminus U)}$ 
is $\Leb^n_\Real$-negligible
\rom(and, therefore, $\st^{-1}(E)$ is $\Leb^n$-negligible\rom).
\item[ii)] $f$ and $\cl f$ are $\Cone$ on  $U\setminus\st^{-1}(E)$ and $C$, respectively. 
\item[iii)]For every $x\in U$ with $\st(x)\in C$ we have $\st(f(x))=\cl f(\st(x))$. Moreover, $\Diff f$ is finite and $\Diff(\cl f)(\st x) = \st(\Diff f(x))$.
\item[iv)] 
\[
\int_U f = \int_{C} \cl f.
\]
\end{enumerate}
\end{lem}

\begin{proof}
By cell decomposition, we may assume that $f$ is a function of class $\Cone$, 
and that $U$ is an open cell. Since $\dim(\Gamma(f))=n$, we have,
by Remark~\ref{rem:HPP}, $\dim(\st(\Gamma(f))\leq n$.  
By cell decomposition, there is an
$\Real_K$-definable, closed, negligible set $E\subset\st(U)$, and definable
functions  $g_k:\st(U)\setminus E\to \Real$ of class $\Cone$ for
$k=1,\dots, r$ such that $\st(\Gamma(f))\cap((\st(U)\setminus E)\times\Real)$
is the union of  the graphs of the functions $g_i$. 
We claim that $r=1$:\\
In fact, if $g_1,\ g_2$ are two different such functions, and say $g_1<g_2$,
then for some $x\in\st(U)$ we have 
$\pair{x,g_1(x)},\, \pair{x,g_2(x)} \in \st(\Gamma(f))$.
Since $f$ is continuous,
$\set{\pair{x,y}:y\in(g_1(x),g_2(x))}\subset \st(\Gamma(f))$. 
On the other hand, $\set{\pair{x,y}: \pair{x,y}\in\st(\Gamma(f))}$ 
is the finite set $\set{\pair{x,g_1(x)},\dots,\pair{x,g_r(x)}}$, absurd.

By \cite[Theorem~10.4]{HPP}, after enlarging $E$ by a negligible set, we
obtain i).

Let $\cl f:=g_1$.  
ii) holds, and for every $x\in U$ with $\st(x)\in C$ we
have $\st(f(x))=\cl f(\st(x))$. 
The equality of the integrals in iv) follows from Remark~\ref{rem:volume}.
To obtain the second part of iii) we will enlarge $E$ by a negligible set. 
For $i = 1,\dots, n$ let 
\[
E_i := \st\bigl(\bigl\{x\in U:\frac{\partial f}{\partial x_i}(x)\notin\Kb\bigr\}\bigr).
\]
By~\cite{bp}, $E_i$ is  $\Real_K$-definable. If $\dim(E_i)=n$, then $E_i$
contains an open ball. This contradicts Lemma~2.5 of~\cite{bo} by which every
definable, one variable function into $\Kb$ has finite derivative except on
$\st^{-1}(A)$, for a finite set~$A$. 
It follows that each set $E_i$ is negligible and therefore, after enlarging $E$, we may assume that $\Diff(f)$ is finite on $U\setminus\st^{-1}(E)$. 

It remains to prove $\Diff(\cl f)(\st x) = \st(\Diff f(x))$. As before, we
will enlarge $E$ by a negligible set.  Let  $V:=\{x\in\Real^n: \Diff(\cl f)(x)
\neq \cl{\Diff f}(x)\}$.    The set $V$ is  $\Real_K$-definable. If $V$ is
non-negligible, then it contains an open ball and therefore w.l.o.g. we may
assume that $V$ is an open ball centered at $0$. We may also assume
$f(0)=0$. After substracting from $f$  a linear function, we can assume that
$\frac{\partial f}{\partial x_i}(0)=0$ and $\frac{\partial\cl f}{\partial
  x_i}(0)=3\epsilon>0$ for some index $i=1,\dots,n$.  Therefore, on a smaller
neighborhood of $0$,  we have $\frac{\partial f}{\partial x_i}<\epsilon$ and
$\frac{\partial\cl f}{\partial x_i}>2\epsilon$. Thus, for $x$ along the $x_i$
axis, $|f(x)|<|x|\epsilon$ and $\cl f(x)\geq2|x|\epsilon$ contradicting the
first part of iii), namely, $\st(f(x))=\cl f(x)$. We conclude that $V$ is
negligible. Let $E'$ be a negligible set such that away from $\st^{-1}(E')$
the equality $\st(\Diff f(x))=\cl{\Diff f}(\st x)$ holds. Then away from
$\st^{-1}(V \cup E')$ we have $\st(\Diff f(x))=\cl{\Diff f}(\st(x))=\Diff\cl f(\st(x))$ as wanted. By cell decomposition, $E$ can be further enlarged so that $C$ is open.
\end{proof}

\begin{rem}
If $f^{-1}(A)$ is negligible whenever $A$ is, then, outside a negligible closed set,
$\cl{(f \comp g)} = \cl f \comp \cl g$.
\end{rem}

\begin{proof}[Proof of Lemma~\ref{lem:change-variables}]
The fact that $f$ is \bilip implies that 
$\cl f$ is injective (since it is also \bilip).
\begin{claim}\label{cl:change-variables-1}
Let $C \subset \st(V)$ be Lebesgue measurable.
Then,
\[
\Leb^n(C) = \int_{(\st f)^{-1}(C)} \st(\abs{\det \Diff f}).
\] 
\end{claim}
In fact, by the change of variables formula (on the reals!) and Lemma~\ref{lem:int-Stan},
\[
\Leb^n(C) = \int_{\cl f^{-1}(C)} \abs{\det \Diff \cl f}) =
\int_{(\st f)^{-1}(C)} \st(\abs{\det \Diff f}). 
\]

\begin{claim}\label{cl:change-variables-2}
Let $h : V \to \Rb$ be an integrable  function.
Then,
\[
\int_{V} h = \int_U \st(\abs{\det \Diff f})\; h \comp f .
\]
\end{claim}
Claim~\ref{cl:change-variables-1} implies that the statement is true if $h$ is a simple
function.
By continuity, the statement is true for any integrable function~$h$.

In particular, we can apply Claim~\ref{cl:change-variables-2} to the function
\[
h: x \mapsto
\begin{cases}
\st(g(x)) & \text{if } x \in f(A),\\
0 & \text{otherwise},
\end{cases}
\]
and obtain the conclusion.
\end{proof}

\begin{lem}[Fubini's theorem]
$\Leb^{n+m}$ is the completion of the product measure $\Leb^n \times \Leb^m$.
Therefore, if $D$ is the interval $[0,1] \subset \K$ 
and given $f: \D^{n+m} \to \D$ definable,
\[
\int_{\D^{n+m}} f(x,y) \de \Leb^{n+m}(x,y) =
\int_{\D^m} \!\! \int_{\D^n} f(x,y) \de \Leb^m(x) \de \Leb^n(y).
\]
\end{lem}
\begin{proof}
Follows from the definition of $\Leb^n$ in~\cite{bo}.
\end{proof}

\subsection{Measure on semialgebraic sets}
\label{sec:algebraic-measure}
\begin{dfn}
We say that $E \subseteq \K^n$ is \zsalg if $E$ is definable without
parameters in the language of pure fields.  
If $E \subseteq \K^n$ is \zsalg
we denote the subset of $\Real^n$ defined by the same formula that defines  $E$ by $E_\Real$. 
\end{dfn}
\begin{rem}
Let $E \subseteq \Kb^n$ be \zsalg.
Then, $\st(E) = \cl{E_\Real}$.
\end{rem}

Let $E \subseteq \K^n$ be closed and \zsalg submanifold.
Working in local charts, from~\cite{bo} one can easily define a measure $\LEK$
on the $\sigma$-ring generated by the definable subsets of $E$ of bounded
diameter. 
We will denote in the same way the completion of $\LEK$.
Notice that $\Leb^{\K^n}= \LnK$.

\begin{rem}
Let $E$ be a closed, \zsalg submanifold of $\K^n$ 
of dimension~$e$, $F := \st(E)$, and $C \subseteq E$ be definable and bounded.
Then, $\LEK(C) = \Leb_\Real^F(\st(C))$, where $\Leb_{\Real}^F$ is the $e$-dimensional
Hausdorff measure on~$F$.
\end{rem}
One could also take the above remark as the definition of $\LEK$ 
on $E \cap \Kb^n$.

\section{Rectifiable partitions}
\label{sec:partition}

Theorem \ref{decomposition} shows that every definable  set $A\subset\Kb^n$
has a partition into definable sets which are $M_n$-cells after an orthonormal
change of coordinates (where $M_n\in\Q$ depends only on $n$). 
In \cite{pawlucki}, the author shows that a permutation of the coordinates suffices. The proof of \ref{decomposition} follows closely that of \cite{kurdyka}. The partition in \ref{decomposition} is  then used in Corollary \ref{cor:partition} to show that definable sets have a rectifiable partition.

\begin{dfn}
Let $L: V \to W$ be a linear map between normed $\K$-vector spaces.   The norm
of $L$ is given by
\[
\norm L := \sup_{\abs v = 1}\abs{L(v)}.
\]
\end{dfn}

For $V,W$ in the Grassmannian of $e$-dimensional linear subspaces of $\K^n$,
namely $ \mathcal G_e(\K^n)$, let $\pi_V$ and $\pi_W\in\End_{\K}(\K^n)$ be the
orthogonal projections onto $V$ and $W$ respectively. 
In this way we have a canonical embedding $\mathcal
G_e(\K^n)\subset\End_{\K}(\K^n)$. 
The \textbf{distance function} on the Grassmannian is given by the inclusion
above: 
\[
\delta(V,W) := \norm{\pi_V-\pi_W}.
\] 

 For $P$ in $\mathcal G_1(\K^n)$ and $X\in \mathcal G_k(\K^n)$, define
\[
\delta(P,X) := \abs{v-\pi_X(v)},
\] 
where $\pi_X$ is the orthogonal projection onto $X$, and $v$ is a generator of $P$ of norm 1.
Note that $\delta(P,X)=0$ if and only if $P\subset X$, $0\leq\delta(P,X)\leq 1$ and $\delta(P,X)=1$ if and only if $P\perp X$.
Note also that $\delta(P, X)$ is the definable analogous of the sine of the
angle between $P$ and~$X$.

\begin{lem}\label{epsilon}
Let $n\in\mathbb{N}_{>0}$. Then there exists an
$\epsilon_n\in\mathbb{Q}_{>0}$, $\epsilon_n <1$, such that for any
$X_1,\dots,X_{2n}\in \mathcal G_{n-1}(\K^n)$, there is a line $P\in \mathcal
G_1(\K^n)$ such that whenever $Y_1,\dots,Y_{2n}\in \mathcal G_{n-1}(\K^n)$ 
and  
\[\begin{aligned}
\delta(X_i,Y_i) &< \epsilon_n, \quad i=1,\dots,2n, \quad \text{then}\\
\delta(P,Y_i)   &> \epsilon_n, \quad i=1,\dots,2n.
\end{aligned}\]
\end{lem}
\begin{proof} 
For $\epsilon>0$ define $S_i(\epsilon)=\{v\in S^{n-1} : |v-\pi_{X_i}(v)| \leq
2\epsilon\}$. 
If $\K=\mathbb{R}$, let $\epsilon_n\in\mathbb{Q}_{>0}$ be small enough so that
$2n \Vol(S_1(\epsilon_n)) < \Vol(S^{n-1})$, where $\Vol$ is the measure
$\Leb^{S^{n-1}}$ defined in \S\ref{sec:algebraic-measure}.
Then
\[
\Vol(\bigcup_{i=1}^{2n} S_i(\epsilon_n)) \leq 2n \Vol(S_1(\epsilon_n)) <
\Vol(S^{n-1})
\]
 and therefore
\begin{equation}\label{don't cover S^n}
\bigcup_{i=1}^{2n}S_i(\epsilon_n)\not=S^{n-1}.
\end{equation}
The same $\epsilon_n$ will necessarily satisfy (\ref{don't cover S^n}) for any field $\K$ containing $\mathbb{R}$. 

Now, we choose
\[
v\in S^{n-1}-\bigcup_{i=1}^{2n}S_i(\epsilon_n)
\]
 and let $P:=\langle v\rangle$. Then
\[
\delta(P,Y_i)=|v-\pi_{Y_i}v|\geq|v-\pi_{X_i}v|-|\pi_{X_i}v-\pi_{Y_i}v|>\epsilon_n.
\qedhere
\]
\end{proof}

\begin{dfn}
Let $\epsilon>0$. A definable embedded submanifold $M$ of $K^n$ is \textbf{$\epsilon$-flat} if for each $x,y\in M$ we have $\delta(TM_x,TM_y)<\epsilon,$ where $TM_x$ denotes the tangent space to $M$ at $x$.
\end{dfn}

\begin{lem}\label{flat} 
Let $A\subset\K^n$ be a definable submanifold of dimension $e$ and
$\epsilon\in\Real_{>0}$. Then there is a cell decomposition $A =
\bigcup_{i=0}^k A_i$ of $A$ such that for every $i$ we have either $\dim(A_i)
< \dim(A)$ or $A_i$ is an $\epsilon$-flat submanifold of~$\K^n$. 
\end{lem}
\begin{proof} 
Cover $\mathcal G_e(\K^n)$ by a finite number of balls $B_i$ of radius $\epsilon/2$; and consider the Gauss map $G:A\to\mathcal G_e(\K^n)$ taking an element $a$ of $A$ to $TA_a$. Take a cell decomposition of $\K^e$ compatible with $A$ and partitioning each $G^{-1}(B_i)$. Then the $e$-dimensional cells contained in $A$ are $\epsilon$-flat. 
\end{proof}
 
 \begin{lem}\label{top dimension}Let $\epsilon\in\mathbb{Q}_{>0}$, and let $A\subset\Kb^n$ be an open  
 definable set. Then there are open, pairwise disjoint cells $A_1,\dots,A_p\subset A$ such that 
\begin{itemize}
  \item[\rm{(i)}]$\dim(A - \bigcup A_i)<n.$
  \item[\rm{(ii)}]For each $i$, there are definable, pairwise disjoint sets $B_1,\dots,B_k$ (with $k$ depending on $i$) such that
  \begin{itemize}
    \item[\rm{(a)}]$k\leq 2n$;
    \item[\rm{(b)}]each $B_j$ is a definable subset of $\partial A_i$ and  an
    $\epsilon$-flat, 					
    $(n-1)$-dimensional, $\Cone$-submanifold of $\K^n$;
    \item[\rm{(c)}]$\dim(\partial A_i - \bigcup_{j=1}^k B_j)<n-1$.
  \end{itemize}
\end{itemize}
\end{lem}
\begin{proof}By induction on $n$. The lemma is clear for $n=1$. Assume that $n>1$ and the  lemma holds for smaller values of $n$. 

Take a cell decomposition of  $\cl A$ compatible with $A$ into $\Cone$-cells. 
Let $C$ be an open
cell in this decomposition; it suffices to prove the lemma for $C$. Note that
$C=(f,g)_X$, where $X$ is an open cell in $\K^{n-1}$ and $f,g$ are definable
$\Cone$-functions on $X$. Take finite covers of $\Gamma(f)$ and $\Gamma(g)$ by
open, definable sets $U_i$ and $V_j$, respectively, such that each
$U_i\cap\Gamma(f)$ and each $V_j\cap\Gamma(g)$ is $\epsilon$-flat (to do this,
take a finite cover of the Grassmannian by $\epsilon$-balls and pull it back
via the Gauss maps for $\Gamma(f)$ and $\Gamma(g)$). The collection of all
sets $\pi(U_i)\cap\pi(V_j)$ is an open cover $\mathcal{O}$ of $X$. By the cell
decomposition theorem, there is  a $\Cone$-cell decomposition of $X$
partitioning each set in $\mathcal{O}$. Let $S$ be an open cell in this
decomposition, and let $C_0:=(f,g)_S$. It suffices to prove the lemma for
$C_0$.  By the inductive hypothesis, we can find $A_1',\dots,A_p'\subset S$
and $B_1',\dots,B_k'\subset\partial A_i'$ satisfying the conditions (i) and
(ii) above (with $n$ replaced by $n-1$). Define
\[
A_i:=(f,g)_{A_i'},\qquad i=1,\dots,p.
\]
Then $\dim(C_0-\bigcup_{i=1}^pA_i)<n$. 
For $j=1,\dots,k$, the set $(B_j'\times\K)\cap\partial A_i$ is definable. Take
a $\Cone$-cell decomposition of this set, and let $B_j$ be the union of the
$(n-1)$-dimensional cells in this decomposition (note that $B_j$ may be
empty). Then $B_j$ is an $\epsilon$-flat $\Cone$-submanifold of $\K^n$ and
\[
\dim\Pa{((B_j'\times\K)\cap\partial A_i)-B_j}<n-1.
\]
Define $B_{k+1}:=\Gamma(f\big| A_i')$ and $B_{k+2}:=\Gamma(g\big| A_i')$; by construction these are $\epsilon$-flat.
It is routine to see that $\partial A_i\subset B_{k+1}\cup B_{k+2}\cup(\partial A_i'\times\K)$.
Thus
\begin{align*}
\partial A_i-\bcup_{j=1}^{k+2}B_j & \subset((\partial
A_i'\times\K)\cap\partial A_i)- \bcup_{j=1}^kB_j\\
&=(\bcup_{j=1}^k((B_j'\times\K)\cap\partial A_i)\cup E)-\bcup_{j=1}^kB_j\\
&\subset\bcup_{j=1}^k(((B_j'\times\K)\cap\partial A_i)-B_j)\cup E,
\end{align*}
where $E$ is a definable set with $\dim(E)<n-1$. 
Therefore $\dim(\partial A_i-\bigcup_{j=1}^{k+2}B_j)<n-1$.
Since $k\leq 2 (n-1)$, we get $k+2\leq 2 n$ and the lemma is proved.
\end{proof}

\begin{dfn}
Let $U \subseteq \K^n$ be open and let  $f: U \to \K^{m}$ be definable. 
Given $0 < M \in \K$, we say that $f$ is an  $M$-function if 
$\abs{\Diff  f}\leq M$.  
We say that $f$ 
has finite derivative if $\abs{\Diff f}$ is finite. 
\end{dfn}
Notice that, by $\omega$-saturation of~$\K$, if $f$ is definable and has finite derivative, then it is an $M$-function for some finite~$M$.

Let $M\in\K_{>0}$. An $M$-cell is a $\Cone$-cell where the $\Cone$ functions that define the cell are $M$-functions. More  precisely:

\begin{dfn}\label{M-cell}
Let $(i_1,\dots,i_m)$ be a sequence of zeros and ones, and $M\in\K_{>0}$. An $(i_1,\dots,i_m)$-$M$-cell is a subset of $\K^m$ defined inductively as follows:
\begin{itemize}
 \item[\rm{(i)}] A (0)-$M$-cell is a point $\{r\}\subset\K$, a (1)-$M$-cell is an interval  		$(a,b)\subset\K$, 	where $a,\ b\in\K$.
 \item[\rm{(ii)}] An $(i_1,\dots,i_m,0)$-$M$-cell is the graph $\Gamma(f)$ of a definable $M$-function $f:X\to\K$ of class $\Cone$,  where $X$ is an $(i_1,\dots,i_m)$-$M$-cell; an 	$(i_1,\dots,i_m,1)$-$M$-cell is a set 
\[
(f,g)_X:=\{(x,r)\in X\times\K:f(x)<r<g(x)\},
\]
 where $X$ is an $(i_1,\dots,i_m)$-$M$-cell and $f,g:X\to\K$ are definable $M$-functions of class $\Cone$ on $X$ such that for all $x\in X$, $f(x)<g(x)$.
\end{itemize}
\end{dfn}

\begin{thm}\label{decomposition}
Let $A\subset\Kb^n$ be definable. Then there are definable, pairwise disjoint sets 
$A_i$, $i=1,\dots,s,$ such that $A=\bcup_i A_i$ and for each $A_i,$ there is a change of coordinates 	$\sigma_i\in O_n(\K)$ such that $\sigma_i(A_i)$ is an $M_n$-cell, where $M_n\in\mathbb{Q}_{>0}$ is a constant depending only on $n$.
\end{thm}

\begin{proof} We will make use of the following fact:

Let $\epsilon\in[0,1]$, $P\in \mathcal G_1(\K^n)$, $X\in \mathcal G_k(\K^n)$ and and $w\in X$ be a unit vector. Suppose $\delta(P,X)>\epsilon$. If $\pi_P(w)\geq 1/2$, where $\pi_P$ is the orthogonal projection onto $P$, then 
\begin{equation*}
|\pi_P(w)-w|\geq|\pi_P(w)-\pi_X(\pi_P(w))|>|\pi_P(w)|\epsilon\geq1/2\epsilon.
\end{equation*}
 If $\pi_P(w)<1/2$, then $|w|\leq|\pi_P(w)|+|\pi_p(w)-w|\leq1/2+|\pi_p(w)-w|$. In either case, we have  
\begin{equation}\label{delta(P,X)}
|\pi_P(w)-w|\geq\frac{1}{2}\epsilon.
\end{equation}

We prove the theorem by induction on $n$; for $n=1$ the theorem is clear. 
We assume  that $n>1$ and that the theorem holds for smaller values of $n$. We also proceed by induction on $d:=\dim(A)$. It's clear for $d=0$; so we assume that $d>0$ 
and the theorem holds for definable bounded subsets $B$ of $\K^{n}$ with $\dim(B)< d$.

Case I: $\dim(A)=n$. In this case $A$ is an open, bounded, definable subset of $\K^n$, so by using the inductive hypothesis and Lemma \ref{top dimension}, we can reduce to the case where there are pairwise disjoint, definable $B_1,\dots,B_k\subset\partial A$ such that $k\leq 2n$, $\dim(\partial A-\bcup_{j=1}^k B_j)<n-1$ and each $B_j$ is an $\epsilon_n$-flat submanifold, where $\epsilon_n$ is as in Lemma~\ref{epsilon}. By Lemma~\ref{epsilon}, there is  a hyperplane $L$ such that for each $B_j$ and all $x\in B_j$, we have $\delta(L^{\perp},T_xB_j)>\epsilon_n$. Take a cell decomposition $\mathcal B$ of $\K^n$, with respect to orthonormal coordinates in the $L$, $L^\perp$ axis, partitioning each $B_j$. Let 
\[
\mathcal S:=\{C\in\mathcal B:\dim(C)=n-1, C\subset\bcup_{j=1}^k B_j\}
\]
and note that $\dim(\partial A\setminus\bcup_{C \in \mathcal S}C)<n-1$. Furthermore, 
\[
\BAD := \{x\in A:\pi_L^{-1}(\pi^{\mathstrut}_L(x)) \cap \partial A \not \subset 
\bcup_{c \in \mathcal S} C\}
\]                                
has dimension smaller than~$n$. 
Let $U_1,\dots,U_l$ be the elements of $\{\pi_L(C):C\in\mathcal S\}$. Then the set 
\[
\{x\in A:x\not\in\pi_L^{-1}(\bcup_{i=1}^l U_i)\}
\]
is contained in~$\BAD$, and therefore has dimension smaller than~$n$. 

By using the inductive hypothesis, we only need to find the required partition
for each of the sets $A\cap\pi_L^{-1}(U_i),$ $i=1,\dots,l$. Fix
$i\in\{1,\dots,l\}$ and let $U:=U_i$, $A':=A\cap\pi_L^{-1}(U)$. 
Take $C\in\mathcal S$ with $\pi_L(C) = U$. 
Then  $C=\Gamma(\phi)$ for a definable $\Cone$-map $\phi:U\to L^\perp$ and for all $x\in C$,
\[
T_xC=\{(v,\Diff\phi(v)):v\in T_{\pi_L(x)}U\} .
\]
Let $v\in T_{\pi_L(x)}U$ be a unit vector; since $\delta(L^\perp,T_xC)>\epsilon_n$ and $|(v,\Diff\phi(v))|=\sqrt{1+|\Diff\phi(v)|^2},$ it follows from equation (\ref{delta(P,X)}) that 
\[
\frac{1}{2}\epsilon_n\leq\frac{1}{\sqrt{1+|\Diff\phi(v)|^2}}|\pi_{L^\perp}((v,\Diff\phi(v)))-(v,\Diff\phi(v))|=\frac{1}{\sqrt{1+|\Diff\phi(v)|^2}}|v|.
\]
Therefore, 
\[
|\Diff\phi(v)|\leq\sqrt{\frac{4}{\epsilon_n^2}-1}.
\]
Let $M_n\in\mathbb{Q}$ be bigger than
$\max\left\{M_{n-1}, \sqrt{\frac{4}{\epsilon_n^2}-1}\right\}.$

We have proved that for each  $C_j\in\mathcal S$ with $\pi_L(C_j)=U$ there is
a definable $\Cone$-map $\phi_j:U\to\K$, such that $|\Diff\phi_j|<M_n$ and 
$C_j = \Gamma(\phi_j)$. 

By the inductive hypothesis, there is a partition $\mathcal P$ of $U$ such that each piece $P\in\mathcal P$ is a $M_{n-1}$-cell after a change of coordinates of $L$. We have 
\[   
A' = \coprod\limits_{\substack{P\in\mathcal P\\(\phi_r,\phi_s)_P\subset A'}} 
(\phi_r,\phi_s)_P,
\]
and $(\phi_r,\phi_s)_P$ is a $M_n$-cell after a coordinate change.

Case II: $\dim(A)<n$. In this case, by Lemma \ref{flat}, we can partition  $A$ into cells which are $\epsilon_n$-flat. Therefore we may assume that  $A$ is an $\epsilon_n$-flat submanifold, where $\epsilon_n$ is as in Lemma \ref{epsilon}. As in case I, there is a hyperplane $L$ such that $A$ is the graph of a  function $f:U\to\K$, $U\subset L$ and  $|\Diff f|<M_n$. By the inductive hypothesis, we can partition $U$ into $M_{n-1}$-cells. The graphs of $f$ over the cells in this partition give the required partition of $A$.
\end{proof}

\begin{dfn}\label{def:e-rect}
Let $A \subseteq \K^n$ and $e \leq n$.
$A$~is \berect with bound~$M$ if, after a permutation of coordinates, 
$A$~is the graph of an $M$-function  $f: U\to \K^{n-e}$, 
where $U \subset \K^e$ is an open $M$-cell for some finite~$M$.
\end{dfn}

\begin{lem}\label{lem:rectifiable-cell}
Let $A \subset \Kb^n$ be an $M$-cell of dimension~$e$.
Then, $A$ is a \berect set, and the bound of $A$ can be chosen 
depending only on $M$ and~$n$.
\end{lem}
\begin{proof}
We proceed by induction on~$n$.
If $n = 0$ or $n = 1$ the result is trivial, so assume $n \geq 2$.
By definition, there exists an  $M$-cell $B \subset \Kb^{n-1}$ such that 
\begin{enumerate}
\item[(1)] either $A = \Gamma(g)$ for some $M$-function $g: B \to \Kb$, or
\item[(2)] $A = (g, h)_B$ for some $M$-functions $g, h: B \to \Kb$, with $g < h$.
\end{enumerate}
By inductive hypothesis, there exists an open  $L$-cell $C \subset \K^d$ (for some~$d$ and some
$L\geq M$ depending only on $M$ and on~$n$), and an $L$-function
$f: C \to \K^{n-1-d}$, such that $B = \Gamma(f)$.

In case (1) $d=e$. Define $l : C \to \K^{n-e}$ by $l(x)=\pair{f(x),g(x,f(x))}$. It is easy to see that
$l$ is an $L'$-function for some $L'$ depending only on $M$ and~$n$, 
and that $A = \Gamma(l)$.

In case (2), $d=e-1$. Define $\tilde g :=  g \circ f$, $\tilde h:= h \circ f$, and
$\tilde B := (\tilde g, \tilde h)_{C}$.
Given $\pair{\x, y} \in \tilde B$, define $l(\x, y) := f(\x)$.
We have that $\tilde B$ is an open $e$-dimensional $L$-cell, $l: \tilde B \to \K^{n - e}$ is an
$L$-function, and $A = \Gamma(l)$.
\end{proof}

\begin{cor}\label{cor:partition}
Let $A \subseteq \K^n$ be definable of dimension at most~$e$. 
Then there is a partition $A=\bcup_{i=0}^kA_i$  such that $\dim(A_0)<e$ and
$A_i$ is a \berect set for $i>0$.
Moreover, the bounds of each $A_i$ can be  chosen to depend only on~$n$ (and
not on~$A$).
We call $(A_0, \dotsc, A_k)$ a \berect partition  of~$A$.
\end{cor}
\begin{proof}
Apply Theorem~\ref{decomposition} and~\ref{lem:rectifiable-cell}.
\end{proof}

Notice that a similar result has also been proved in \cite[Theorem~2.3]{PW} (where they also take arbitrarily small bounds): 
however, in \cite{PW} they don't require that the 
functions parametrizing the set $A$ are injective (which is essential for our
later uses).


\section{Whitney decomposition}\label{sec:whitney}

The fact that the functions that define an $M$-cell are actually Lipschitz
function follows from the following property of $M$-cells: 
\begin{sentence}
Every pair of
points $x,y$ in an $M$-cell $C\subset \K^n$ can be connected by a definable
$\Cone$ curve $\gamma:[0,1]\to C$ with $|\gamma'(t)| < N |x-y|$, where $N$ is a
constant depending only on $M$ and $n$ which is finite if $M$ is (Lemma
\ref{bdd.derivative} or \cite{evr} 3.10 \& 3.11). 
\end{sentence}
The same property implies
that a $N$-function $f$ on an $M$-cell is Lipschitz where the Lipschitz
constant is finite if $M$ and $N$ are (Corollary  \ref{lipschitz}). This last
property will be needed for defining Hausdorff measure. 

\begin{rem}
Let $U \subset \Kb^n$ be open and definable, and $f: U \to \Kb$ be an
$M$-function (for some finite~$M$).
It is not true in general that $f$ is $L$-Lipschitz for some finite~$L$:
this is the reason why we needed to prove Theorem~\ref{decomposition}.
\end{rem}

\begin{dfn}
Let $A\subset\K^n$, $B\subset\K^m$ be definable sets. Let
$\lambda\subset A\times([0,1]\times B)\subset\K^n\times
\K^{1+m}$ be a definable set such that for every $x\in A$, the fiber over $x$ 
\[
\lambda_x:=\{y\in[0,1]\times B: \pair{x,y}\in\lambda\}
\]
is a curve  $\lambda_x:[0,1]\to B$. We view $\lambda$ as describing the family of curves $\{\lambda_x\}_{x\in A}$. Such a family is a definable family of curves (in $B$, parametrized by $A$). 
\end{dfn}

An $L$-cell is an $L$-Lipschitz cell if the functions that define the $L$-cell are $L$-Lipschitz.

\begin{lem}\label{bdd.derivative}Fix $L\in\K_{>0}$ and $n\in\Nat_{>0}$. Then, there is a constant $K(n,L)\in\K_{>0}$ depending only on $n$ and $L$, that is finite if $L$ is, such that for every $L$-Lipschitz cell $C\subset\K^n$ there is  a definable family of curves $\gamma\subset C^2\times([0,1]\times C)$ such that: For all  $x,y\in C$, $\gamma_{x,y}:[0,1]\to C$ is a $\Cone$-curve  with
\begin{itemize}
\item[\rm{(i)}]  $\gamma_{xy}(0)=x,\gamma_{xy}(1)=y$; 
\item[\rm{(ii)}] $|\gamma_{xy}'(t)|\leq K(n,L)|x-y|$, for all $t\in[0,1]$.
\end{itemize}
\end{lem}
\begin{proof}
By induction on $n$. For $n=1$ the lemma is clear. Take $n\geq1$, and assume
that the lemma holds for  $n$. Let $C\subset\K^{n+1}$ be an $L$-Lipschitz
cell. Then $C=\Gamma(f)$ or $C=(g,h)_X$ for some $L$-Lipschitz cell
$X\subset\K^{n-1}$ and definable, $\Cone$, $L$-Lipschitz functions $f,g,h$ with
$g<h$, and $|\Diff f|,|\Diff g|,|\Diff h|\leq L$. By induction, there are a
constant $k:=K(n-1,L)$ and a definable family of $\Cone$-curves $\beta$ in $X$
with the required properties.  Let $\pi_n:\K^{n+1}\to\K^n$ be the projection
onto the first $n$ coordinates. 

If $C=\Gamma(f)$, we lift $\beta$ to $C$ via $f$: fix $x,y\in C$ and let $\gamma_{x,y}(t):=(\alpha(t),f(\alpha(t)))$, where for all $t\in[0,1]$ $\alpha(t):=\beta_{\pi_n(x),\pi_n(y)}(t).$  Then we have   			
$
|\gamma_{xy}'(t)|\leq
(1+L)k|x-y|.
$

If $C=(g,h)_X$, we lift $\beta$ as follows: Fix $x,y\in C$ and let $\alpha:=\beta_{\pi_n(x),\pi_n(y)}$. Let $\pi:\K^{n+1}\to\K$ be the projection onto the last coordinate and take $u,v\in(0,1)$  with 
\begin{align*}\pi(x)&=uh(\alpha(0))+(1-u)g(\alpha(0))\\
		      \pi(y)&=vh(\alpha(1))+(1-v)g(\alpha(1)).
\end{align*}
Let $l(t):=tv+(1-t)u$, for $t\in[0,1]$. We define
$\gamma_{x,y}(t):=(\alpha(t),l(t)h(\alpha(t))+(1-l(t))g(\alpha(t))),$
and note that 
\[
|\gamma_{xy}'(t)|
\leq k|x-y|+|(v-u)(h(\alpha(t))-g(\alpha(t)))|+2Lk|x-y|,
\]
since $l(t), 1-l(t)$ are between 0 and 1 and $|\Diff h(\alpha'(t))|,|\Diff g(\alpha'(t))|\leq L|\alpha'(t)|$.
Let $f := h-g$. We want to bound $|(v-u)f(\alpha(t))|$, which equals
\[
|\pi y - \pi x - v(f(\alpha(1)) - f(\alpha(t))) 
+ u(f(\alpha(0)) - f(\alpha(t))) + g(\alpha(0)) - g(\alpha(1))|.
\]
But
\[
|f(\alpha(1))-f(\alpha(t))| \leq L|\alpha(1) - \alpha(t)| = 
L\abs{1-t}\left|\frac{\alpha(1)-\alpha(t)}{1-t}\right|\leq
L|\alpha'(t_0)|
\]
for some $t_0$ between $t$ and 1. Similarly, $|f(\alpha(0))-f(\alpha(t))|\leq L|\alpha'(t_1)|,$ for some $t_1$ between $t$ and $1$.
Since $u,v\in[0,1]$, we get 
\[
|(v-u)f(\alpha(t))|\leq|\pi y-\pi x|+2Lk|x-y|+L|x-y|;
\]
thus $|\gamma_{xy}'(t)|\leq K(n,L)|x-y|$ for some constant $K(n,L)$ depending only on $n$ and $L$ which is finite if $L$ is. The collection of the curves $\gamma_{xy}$ for $x,y\in C$ constitutes the required family of curves. 
\end{proof}

\begin{thm}\label{K are tame}
Let  $L>0$, and let $C\subset\K^n$ be an $L$-cell. Then $C$ is a
$k(n,L)$-Lipschitz cell, where $k(n,L)$ depends only on $n$ and~$L$, 
and is finite if $L$ is.
\end{thm}
\begin{proof}By induction on $n$; the theorem is clear for $n=1$. Assume that $n>1$ and that the theorem holds for  $n-1$.  Then $C=\Gamma(f)$ or $C=(g,h)_X$,  where $X\subset\K^{n-1}$ is a $k(n-1,L)$-Lipschitz cell and $f,g,h$ are $\Cone$-functions on $X$ such that  $|\Diff f|,|\Diff g|,|\Diff h|\leq L$. We need to show that $f,g,h$ are Lipschitz.

Since $X$ is a $k$-Lipschitz cell, $k:=k(n-1,L)$,  it follows from Lemma \ref{bdd.derivative} that there is a constant $K(n-1,k)$  such that whenever $x,y\in X$, there is a definable, $\Cone$-curve $\gamma$ joining $x$ and $y$ with 
$|\gamma'(t)|\leq K(n-1,k)|x-y|$ for all  $t\in[0,1]$. Let $g:=f\circ\gamma$, and let $t_0\in(0,1)$ be such that
\[
|f(x)-f(y)|=|g'(t_0)|=|\Diff f(\gamma'(t_0))|\leq L|\gamma'(t_0)|\leq
LK(n-1,k)|x-y|.
\]
Thus $f$ is $LK(n-1,k)$-Lipschitz. We set $k(n,L):=LK(n-1,k)$.
\end{proof}

\begin{cor}\label{lipschitz} Let $C$ be an $M$-cell and $f$ be a definable
$M$-function. Then $f$ is Lipschitz, and  with finite Lipschitz constant if
$M$ is finite. 
\end{cor}
\begin{proof} By Theorem \ref{K are tame}, $C$ has a definable family of curves as in Lemma \ref{bdd.derivative}. The result therefore follows from the mean value theorem.
\end{proof}
 
\begin{dfn} A definable set $A\subset\K^n$ satisfies the Whitney arc property if there is a constant $K\in\Kb_{>0}$ such that for all $x,y\in A$ there is a definable curve $\gamma:[0,1]\to A$ with $\gamma(0)=x$, $\gamma(1)=y$ and $\length(\gamma):=\int_0^1|\gamma'|\leq K|x-y|$.
\end{dfn}

\begin{lem}\label{M-WAP}
Let $C \subset \Kb^n$ be an $M$-cell, $M\in\Kb$. Then, $C$ satisfies the Whitney arc property.
\end{lem}
\begin{proof} It follows from Theorem \ref{K are tame} and Lemma \ref{bdd.derivative}.
\end{proof}

\begin{thm}
Let $A \subset \Kb^n$ be definable.
Then, $A$ can be partitioned into finitely many definable sets, each of them
satisfying the Whitney arc property.
\end{thm}
\begin{proof} This follows from Lemma \ref{M-WAP}, Theorem \ref{decomposition} and the fact that the Whitney arc property is invariant under an orthonormal change of coordinates.
\end{proof}


\section{Hausdorff measure}

For an introduction to geometric measure theory, and in particular to
the Hausdorff measure, see \cite{morgan}.

\begin{dfn}
Let $U \subseteq \K^n$ be open and  let $f: U \to \Kb^m$ be a
definable function. If $a \in U$,  $e \leq n$ and 
 $M$ is the set of the $e \times e$ minors of $\Diff f(a)$
we define 
\[
\Je f(a) = \begin{cases} + \infty  & \text{if  $f$ is not differentiable
  at~$a$ or $\rank(\Diff f(a))> e$},\\
\sqrt{\sum_{m \in M} m^2} & \text{otherwise;}
\end{cases}\]
(\cf~\cite[\S 3.6]{morgan}).
\end{dfn}
Notice that if $e = n = m$, then $\Jac_n f = \abs{\det(\Diff f)}$.

\begin{dfn}\label{def:area-basic}
Let $U \subseteq \Kb^e$ be an open $M$-cell for some $M \in \Nat$, and let
$f: U \to \Kb^m$ be a definable function with  finite derivative.
Let $F: U \to \Kb^{m+e}$ be $F(x) := \pair{x, f(x)}$ and
$C := \Gamma(f) = F(U)$ (notice that $C$ has bounded diameter).
We define
\[
\HeK(C) := \int_U \Jac_e F \de \LeK.
\]
\end{dfn}

\begin{lem}\label{lem:Hausdorff-standard}
If $C \subseteq \Kb^n$ is \berect, then $\Haus^e(C) = \He_\Real(\st(C))$,
where $\He_\Real$ is the $e$-dimensional Hausdorff measure on $\Real^n$.
\end{lem}
\begin{proof}
Let $A \subset \Kb^e$ and $f: A \to \Kb^{n-e}$ be as in
Definition~\ref{def:e-rect}, and $F: A \to \Kb^n$ as in Definition~\ref{def:area-basic}.
Let $B := \st(A)$.
Then, using the real Area formula~\cite{morgan},
\[
\int_A \Jac_e F \de \LeK =
\int_B \Jac_e (\cl F)  \de \Leb^e_\Real =
\He_\Real(\cl F(B)) =
\He_\Real(\st(C)).
\]
\end{proof}

\begin{dfn}
Let $A \subseteq \Kb^n$ be definable of dimension at most $e$, and $(A_0, \dotsc, A_k)$ be a \berect
partition of~$A$.
Define 
\[
\Haus^e(A) := \sum_{i = 1}^k \Haus^e(A_i),
\]
where  $\Haus^e(A_i)$ is defined using~\ref{def:area-basic}.
\end{dfn}

\begin{lem}\label{lem:Hausdorff-well-defined}
If $A$ is as in the above definition, then $\Haus^e(A)$ does not depend 
on the choice of the \berect partition $(A_0, \dotsc, A_k)$.
\end{lem}
\begin{proof}
It suffices to prove the following: if $C$ is a \berect set and 
$(A_0, \dotsc, A_k)$ is a \berect partition of~$C$, then 
$\HeK(C) = \sum_{i = 1}^k \HeK(A_i)$, where $\HeK(C)$ and  $\HeK(A_i)$ are
defined using~\ref{def:area-basic}.
For every $i =1, \dotsc, n$ let $U$ and $V_i$ be $M$-cells, 
$f :U \to \K^{n-e}$  and $g_i : V_i \to \K^{n-e}$ be definable functions with
finite derivative, $\sigma_i$ be a permutation of variables of~$\K^n$,
$F : \K^e \to \K^n$ defined by $F(x) := (x, f(x))$, 
and $G_i: \K^e \to \K^n$ defined by $G(x) = \sigma_i(x, g_i(x))$
such that $C = F(U)$ and $A_i = G_i(V_i)$.
Define  $U_i := F^{-1}(A_i) \subseteq U$, and 
$H_i := G_i^{-1} \comp F : U_i \to V_i$.
Notice that each $H_i$ is a bi-Lipschitz bijection, that $U$ is the disjoint union of
the~$U_i$, and that $\dim(U_0) < e$.
Hence,
\begin{multline*}
\HeK(C) = \int_U \Je F \de \LeK = \sum_{i = 1}^n \int_{U_i} \Je F \de \LeK =
\sum_{i = 1}^n \int_{U_i} \Je(G_i \comp H_i) \de \LeK =\\
= \sum_{i = 1}^n \int_{U_i} (\Je(G_i) \comp H_i) \cdot \abs{\det(\Diff H_i)} \de \LeK =
\sum_{i = 1}^n \int_{V_i} \Je G_i \de \LeK =
\sum_{i = 1}^n \HeK(A_i),
\end{multline*}
where we used Lemma~\ref{lem:change-variables}, the fact that 
each $\sigma_i$ is a linear function with determinant~$\pm 1$,
and that $\Je(G \comp H) = (\Je(G) \comp H) \cdot \abs{\det(\Diff H)}$.
\end{proof}

\begin{lem}\label{lem:Hausdorff-n-invariant}
$\Haus^e$ does not depend on~$n$.
That is, let $m \geq n$, and $A \subset \Kb^n$ definable, 
and $\psi: \K^n \to \K^{m}$ be the embedding $x \mapsto (x,0)$.
Then, $\Haus^e(A) = \Haus^e(\psi(A))$.
\end{lem}
\begin{proof}
Obvious from the definition and Lemma~\ref{lem:Hausdorff-well-defined}.
\end{proof}

Notice that $\Haus^0(C)$ is the cardinality of~$C$.

It is clear that $\Haus^e$ can be extended to the $\sigma$-ring generated by
the definable subsets of $\K^n$ of finite diameter and dimension at most~$e$;
we will also denote  the completion of this  extension by $\Haus^e$.

\begin{lem}
$\Haus^e$ is a measure on the $\sigma$-ring generated by
the definable subsets of $\K^n$ of bounded diameter and dimension at most~$e$.
\end{lem}

\begin{proof}
Since $\K$ is $\aleph_1$-saturated, it suffices to show that, for every
$A$ and $B$ disjoint definable subsets of $\K^n$ of finite diameter and
dimension at most~$e$, $\HeK(A \cup B) = \HeK(A) + \HeK(B)$.
But this follows immediately from Lemma~\ref{lem:Hausdorff-well-defined}.
\end{proof}

\begin{example}\label{ex:double}
In Lemma~\ref{lem:Hausdorff-standard}, the assumption that $C$ is \berect is
necessary.
For instance, take $\epsilon > 0$ infinitesimal, and $X$ be the following
subset of $\K^2$
\[
X := \Pa{[0,1] \times \set 0} \cup \set{\pair{x,y}: 0 \leq x \leq 1 \ \&\  
y = \epsilon x}.
\]
Then, $\st(X) = [0,1] \times \set 0$, and thus
$\Haus^1(X) = 2$, while $\Haus^1_\Real(\st(X)) = 1$.
This is the source of complication in the theory, and one of the reasons why
we had to wait until this section to introduce $\HeK$.
\end{example}


\section{Cauchy-Crofton formula}
Give $e \leq n$, define
\[
\beta := \Gamma\Pa{\tfrac{e + 1} 2} \Gamma\Pa{\tfrac{n - e + 1} 2}
\Gamma\Pa{\tfrac{n + 1} 2}^{-1} \pi^{- 1/2}.
\]

\begin{dfn}
Let $\Agr_e(\K^n)$ be the Grassmannian of affine $e$-dimensional subspaces
of $\K^n$ and let  $\Agr_e(\Real^n)$ be the Grassmannian of affine
$e$-dimensional subspaces of~$\Real^n$.
Fix an embedding of $\Agr_e(\Real^n)$ into some $\Real^m$,
such that $\Agr_e(\Real^n)$ is a \zsalg
closed submanifold of~$\Real^m$, 
and the restriction to $\Agr_e(\Real^n)$  of the 
$\dim(\Agr_e(\Real^n))$-dimensional Hausdorff measure 
coincides with the Haar measure on $\Agr_e(\Real^n)$.
\end{dfn}

\begin{dfn}
Given $A \subseteq \K^n$ and $E \in \Agr_{n-e}(\K^n)$, let $f_A(E) := \# (A \cap E)$.
\end{dfn}

\begin{thm}[Cauchy-Crofton Formula]\label{thm:Cauchy-Crofton}
Let $A \subseteq \Kb^n$ be definable of dimension $e$.
Then,
\[
\HeK(A) = \frac 1 \beta \int_{\Agr_{n-e}(\K^n)} f_A \de \Leb^{\Agr_{n-e}(\K^n)}.
\]
\end{thm}

We prove the theorem by reducing it to the known case of $\K=\Real$. This is done by showing that $\# (A\cap E)$ equals $\#(\st A\cap\st E)$ almost everywhere.

\begin{dfn}
Let $f: U \to \Kb^m$ be definable, with $U\subset\Kb^n$ open. 
Let $E \subset \Real^n$ and $\cl f$ be as in Lemma~\ref{lem:int-Stan}.
We say that  
$b \in \Real^n$ is an $S$-regular point of~$\cl f$ if
\begin{enumerate}
\item [i)] $b \in \st(U) \setminus \cl E$;
\item [ii)] $b$ is a regular point of~$\cl f$.
\end{enumerate}
Otherwise, we say that $b$ is an $S$-singular point 
and $\cl f(b)$ is an $S$-singular value of~$\cl f$. 
If $c\in\Real^m$ is not an $S$-singular value, we say that $c$ is an $S$-regular value of $\cl f$.
\end{dfn}

\begin{rem}
Let $S$ be the set of $S$-regular points of $\cl f$.
Then, $S$~is open and definable in~$\Real_\K$. 
\end{rem}

\begin{lem}[Morse-Sard]\label{lem:Sard}
Assume that $m \geq n$.
Then, the set of $S$-singular values of~$\cl f$ is $\Leb^m_\Real$-negligible,
\end{lem}
\begin{proof} 
By Lemma~\ref{lem:int-Stan}, $E$ is negligible;
since $E$ is also $\Real_\K$-definable, it has empty
interior and therefore $\dim(E)<n$. 
Since $m\geq n$, it  follows that $\cl f(E)$ is negligible. 
The set of $S$-singular values of $\cl f$ is the union of
$\cl f(E)$ and the set of singular values of~$\cl f$; it is therefore negligible.
\end{proof}

\begin{lem}[Implicit Function]\label{lem:implicit}
Assume that $m = n$.
Let $b \in \Real^n$.
If $b$ is an $S$-regular point of~$\cl f$ then, for every $y \in \st^{-1}(\cl f(b))$
there exists a unique $x \in \st^{-1}(b)$ such that $f(x) = y$.
\end{lem}
\begin{proof}
Choose $x_0 \in \st^{-1}(b)$.
Let $A := (\Diff f(x_0))^{-1}$. Since $b$ is a regular point of $\cl f$, $\norm A$ is finite.
Thus we can choose $r, \rho \in \Q_{> 0}$ such that $B := \cl{B(b; \rho)}$
is contained in the set of $S$-regular points of~$\cl f$,
and
\[\begin{aligned}
\norm{\Diff \cl f(b') - \Diff \cl f(b)} &< \frac{1}{2 n \norm A}, & \text{for every }
b' \in B\\
r &\leq \frac{\rho}{2 \norm A}.
\end{aligned}\]
Moreover, we can pick $\rho$ such that  $B' := \cl{B(x_0; \rho)} \subset U$.
Given $y \in \K^n$ such that $\abs{y - f(x_0)} < r$, consider the mapping
\[\begin{aligned}
T_y    &: B' \to \K^n\\
T_y(x) &:= x + A \cdot (y - f(x)).
\end{aligned}\]
$T_y$ is definable and Lipschitz, with Lipschitz constant $1/2$.
Therefore, for every $y \in B(f(x_0); r)$ there exists a unique $x \in B'$
such that $T_y(x) = x$. Thus, there is a unique $x\in B$ with $f(x) = y$.
It remains to show that, given $y \in \st^{-1}(\cl f(b))$ and $x \in B'$ such that
$f(x) = y$, we have $x \in \st^{-1}(b)$.
We can verify that
\[\begin{aligned}
\cl T_y &: B \to B\\
\cl T_y(b') &= b' + (\Diff \cl f(b))^{-1} \cdot (\cl f(b) - \cl f(b'))
\end{aligned}\]
is also a contraction, and therefore it has a unique fixed point, namely $b$. Since $\cl T_y(\st(x))=\st(x)$, we must have $\st(x)=b$.
\end{proof}

\begin{rem}\label{rem:lip-zero}
Let $U \subset \Kb^m$. If $f: U \to \Kb^n$ is definable and $M$-Lipschitz (for some
finite~$M$), $n\geq m$ and $E$ is $\Leb_\Real^m$-negligible,
then the set $f(\st^{-1}(E))$ is $\Leb^n$-negligible.
\end{rem}
\begin{proof}
We can cover $E$ with a polyrectangle $Y$ whose measure is an arbitrarily small rational number ~$\lambda$ and such that ~$Y$  covers $\st^{-1}(E)$.  Since $f(Y)$ has measure at most $CM^n \lambda$ (C depends only on $m$ and $n$) the result follows.
\end{proof}

\begin{lem}\label{lem:card}
Let $A \subseteq \Kb^n$ be a \berect set of dimension~$e$.
Consider $V := \K^e$ as embedded in $\K^n$ via the map
$x \mapsto \pair{x, 0}$.
Identify each $p \in V$ with the $(n-e)$-dimensional affine space which is
orthogonal to $V$ and intersects $V$ in~$p$.
Then, for almost every $p \in V$, we have
$\#(p \cap A) = \#(\st(p) \cap \st(A))$.
\end{lem}

\begin{proof}
Let $\pi: \K^n \to V$ be the orthogonal projection.
Let $U \subset \Kb^e$ be an open $M$-cell and $f: U \to \K^{n-e}$ be
a definable $M$-function ($M$ finite) such that $A = \Gamma(f)$. 
Let $F(x) := \pair{x, f(x)}$.
Let $h := \pi \circ F: U \to V$, and consider $\cl h: C\to \st(V)$,
$C\subset\st(U)$ as in Lemma~\ref{lem:int-Stan}. 
For almost every $p \in V$, $\#(p \cap A) = \#(h^{-1}(p))$, 
and $\#(\st p \cap st A) = \#(\cl h^{-1}(\st p))$
because $F: U \to A$ and  $\cl F: C \to \Imag(\cl F)$ are bijections.
Thus, it suffices to prove that, for almost every $p \in V$,
$\#(h^{-1}(p)) = \#(\cl h^{-1}(\st p))$.  Let $E$ be as in Lemma ~\ref{lem:int-Stan}. By Remark \ref{rem:lip-zero}, $h(\st^{-1}(E))$ is $\Leb^e$-negligible. Let $S$ be the set of $S$-singular values of $\cl h$, by Lemma \ref{lem:Sard}, $S$ is negligible.

Let $p \in V \setminus (\st^{-1}(S)\cup~ h(\st^{-1}(E))$. Then for  every $x$ in $h^{-1}(p)$, ~$\st(x)$ is  an $S$-regular point of $\cl h$, and therefore  Lemma~\ref{lem:implicit} implies $\#(h^{-1}(p)) = \#(\cl h^{-1}(\st p))$.
\end{proof}
Notice that the above lemma does not hold if $A$ is only definable, instead of 
\berect.

\begin{proof}[Proof of Theorem~\ref{thm:Cauchy-Crofton}]
By Corollary~\ref{cor:partition}, \wloG $A$ is \berect.
Let $B := \st(A)$, and $f_B(F) := \#(B \cap F)$, for every 
$F \in \Agr_e(\Real^n)$.
By Lemma~\ref{lem:card},
\[
\int_{\Agr_{n-e}(\K^n)} f_A \de \Leb^{\Agr_{n-e}(\K^n)} =
\int_{\Agr_{n-e}(\Real^n)} f_B \de \Leb^{\Agr_{n-e}(\Real^n)}.
\]
By the usual Cauchy-Crofton formula~\cite[3.16]{morgan}, the right-hand side
in the above identity is equal to $\Haus^e_\Real(B) = \Haus^e(A)$, where we
applied Lemma~\ref{lem:Hausdorff-standard}.
\end{proof}


\section[Coarea formula]{Further properties of Hausdorff measure and the
  Co-area formula}

\begin{thm}
Let $e \leq n$ and $C \subseteq \K^n$ be bounded and definable of dimension 
at most~$e$.
\begin{enumerate}
\item $\HeK$ is invariant under isometries.
\item For every $r \in \Kb$, $\HeK(r C) = \st(r)^e \HeK(C)$.
\item If $C$ is \zsalg, then
$\HeK(C) = \He(C_\Real) = \He(\st(C))$.
\item if $\dim(C) < e$, then $\HeK(C) = 0$; the converse is not true.
\item $\HeK(C) < + \infty$.
\item If $\Pa{C(r)}_{r \in \K^d}$ is a definable family of bounded subsets 
of~$\K^n$, then there exists a natural number $M$ such that $\HnK(C(r)) < M$
for every $r \in \K^d$.
\item If $\K'$ is either an elementary extension or an o-minimal expansion
of~$\K$, then $\He(C_{\K'}) = \HeK(C)$.
\item If $n = e$, then $\He(C) = \LnK(C)$.
\item If $C$ is a subset of an $e$-dimensional affine space~$E$, 
then $\HeK(C) = \LEK(C)$.
\end{enumerate}
\end{thm}

\begin{proof}\
\begin{enumerate}
\item[(1)] Use the Cauchy-Crofton formula. 
\item[(2), (4) and (7)] Apply the definition of $\Haus^e$ and Lemma~\ref{lem:Hausdorff-well-defined}.
\item[(3)] Apply Corollary~\ref{cor:partition} to $C_\Real$ and use
Lemma~\ref{lem:Hausdorff-standard}.
\item[(5) and (6)] Apply the Cauchy-Crofton formula: see~\cite{dries03}.
\item[(8)] Apply Lemma~\ref{lem:Hausdorff-standard}.
\item[(9)] Since $\HeK$ is invariant under isometries,
\wloG $E$ is the coordinate space~$\K^e$.
By Lemma~\ref{lem:Hausdorff-n-invariant}, the measure $\HeK$ inside $\K^n$
is equal to the measure $\HeK$ inside $\K^e$, and the latter is equal 
to~$\Leb^e$.
The conclusion follows from Remark~\ref{rem:BO}.
\qedhere
\end{enumerate}
\end{proof}

The following theorem is the adaption to o-minimal structures of the Co-area
formula, a well-known generalization of Fubini's theorem.
Let $\D := [0,1] \subset \K$.

\begin{thm}[Co-area Formula]
Let $A \subset D^m$ be definable,
and $f: \D^m \to \D^n$ be a definable Lipschitz function, with $m \geq n$.
Then, $\Jac_n f$ is $\Leb^m_{\K}$-integrable, and
\[
\int_A \Jac_n f \de \Leb^m =
\int_{\D^n} \Haus^{m - n}(A \cap f^{-1}(y)) \de \LnK(y).
\]
\end{thm}
\begin{proof}[Sketch of Proof]
\Wlog, $A$ is an open subset of~$\D^m$.
By Lemma~\ref{lem:Sard}, \wloG all points of $A$ are $S$-regular for~$\cl f$.
Apply the real co-area formula~\cite{morgan} to $g := \cl f$ and $B :=
\st(A)$, and obtain
\[
\int_A \Jac_n f \de \Leb^m = \int_B \Jac_n g \de \Leb^m_\Real =
\int_{\D_\Real^n} \Haus^{m - n}_\Real (B \cap g^{-1}(z)) \de \Leb^n_\Real(z).
\]
By the Implicit Function Theorem and Lemma~\ref{lem:Hausdorff-standard},
for almost every $y \in \D_{\Real}^n$, we have
\[
\Haus^{m - n}(A \cap f^{-1}(y)) = \Haus^{m - n}_{\Real}(B \cap g^{-1}(\st y)).
\qedhere
\]
\end{proof}



\end{document}